\documentclass{elsart}

\usepackage{graphics}
\usepackage{graphicx}
\usepackage{epsfig}
\usepackage{amsfonts,amssymb}

\newtheorem{postulate}{Hypothesis}[section]  %hypoth\`ese
\newcommand{\bhypo}{\begin{postulate} \it}
\newcommand{\ehypo}{\end{postulate}}

\newtheorem{theorem}{Theorem}[section]  %theoreme
\newcommand{\btheo}{\begin{theorem} \it}
\newcommand{\etheo}{\end{theorem}}

\newtheorem{proposition}{Proposition}[section]  %proposition
\newcommand{\bprop}{\begin{proposition}\it}
\newcommand{\eprop}{\end{proposition}}

\newtheorem{corollary}{Corollary}[section]  %corollaire
\newcommand{\bcorol}{\begin{corollary} \it}
\newcommand{\ecorol}{\end{corollary}}

\newtheorem{lemma}{Lemma}[section]  %lemme
\newcommand{\blem}{\begin{lemma}\it}
\newcommand{\elem}{\end{lemma}}

\newtheorem{remark}{Remark}[section]   %remarque
\def\brem{\begin{remark} \rm }
\def\erem{\end{remark}}

\newtheorem{definition}{Definition}[section]   %definition
\newcommand{\bdefi}{\begin{definition} \rm }
\newcommand{\edefi}{\end{definition}}

\date{}

\newcommand{\ba}{\begin{array}}
\newcommand{\ea}{\end{array}}
\newcommand{\bea}{\begin{eqnarray}}
\newcommand{\eea}{\end{eqnarray}}
\newcommand{\bead}{\begin{eqnarray*}}
\newcommand{\eead}{\end{eqnarray*}}
\newcommand{\be}{\begin{equation}}
\newcommand{\ee}{\end{equation}}
\newcommand{\bed}{\begin{displaymath}}
\newcommand{\eed}{\end{displaymath}}
\newcommand{\bl}{\begin{lemmma}}
\newcommand{\el}{\end{lemmma}}
\newcommand{\bp}{\begin{proposition}}
\newcommand{\ep}{\end{proposition}}
\newcommand{\bt}{\begin{theorem}}
\newcommand{\et}{\end{theorem}}
\newcommand{\br}{\begin{remark}}
\newcommand{\er}{\end{remark}}
\newcommand{\bc}{\begin{corollary}}
\newcommand{\ec}{\end{corollary}}
\newcommand{\bd}{\begin{definition}}
\newcommand{\ed}{\end{definition}}

\newenvironment{proof}%
{\begin{sloppypar}\noindent{\bf Proof }}%
{\hspace*{\fill}$\square$\end{sloppypar}}

\begin{document}

\begin{frontmatter}

\title{Quasi-Sectorial Contractions}

\author{Valentin A. Zagrebnov}
\address{Universit\'{e} de la M\'{e}diterran\'{e}e (Aix-Marseille II) and
Centre de Physique Th\'{e}orique - UMR 6207,  Luminy-Case 907, 13288 Marseille Cedex 9, France}
\ead{zagrebnov@cpt.univ-mrs.fr}

\begin{abstract}
We revise the notion of the \textit{quasi-sectorial} contractions.
Our main theorem establishes a relation between semigroups of \textit{quasi-sectorial} contractions and
a class of $m-${sectorial} generators. We discuss a relevance of this kind of contractions to
the theory of operator-norm approximations of strongly continuous semigroups.
\end{abstract}


\end{frontmatter}

%%%%%%%%%%%%%%%%%%%%%%%%%%%%%%%%%%%%%%%%% Section 1 %%%%%%%%%%%%%%%%%%%%%%%%%%%%%%%%%%%%%%%%%%%%%%%%%%%%
\section{Sectorial Operators}\label{sec1}
\setcounter{equation}{0}
\renewcommand{\theequation}{\arabic{section}.\arabic{equation}}

Let $\mathfrak{H}$ be a separable Hilbert space and let $T$ be a densely defined linear
operator with domain ${\rm{dom}}(T)\subset \mathfrak{H}$.
%%%%%%%%%%%%%%%%%%%%%%%%%%%%%%%%%%%%%%%%%%%%%%%%%%%%%%%%%%%%%%%%%%%%%%%%%%%%%%%%%%%%%%%%%%%%%%%
\begin{definition}\label{def-n-ran}
{\em The set of complex numbers:
$$
{\mathfrak{N}}(T) := \{(u,Tu)\in {{\mathbb{C}}}: u \in {\rm{dom}}(T), \
\|u\| = 1\} ,
$$
is called the \textit{numerical range} of the operator $T$.}
\end{definition}
%%%%%%%%%%%%%%%%%%%%%%%%%%%%%%%%%%%%%%%%%%%%%%%%%%%%%%%%%%%%%%%%%%%%%%%%%%%%%%%%%%%%%%%%%%%%%%%%
\begin{remark}\label{rem-1-1}
{\rm{(}}a{\rm{)}} It is known that the set ${\mathfrak{N}}(T)$ is convex
{\rm{(}}the Toeplitz-Hausdorff theorem{\rm{)}}, and in general is
neither open nor closed, even for a closed operator $T$.\\
{\rm{(}}b{\rm{)}} Let $\Delta:= {{\mathbb{C}}} \setminus
\overline{{\mathfrak{N}}(T)}$ be complement of the numerical range
closure in the complex plane. Then $\Delta$ is a connected open set
except the special case, when $\overline{{\mathfrak{N}}(T)}$ is a
strip bounded by two parallel straight lines.
\end{remark}
%%%%%%%%%%%%%%%%%%%%%%%%%%%%%%%%%%%%%%%%%%%%%%%%%%%%%%%%%%%%%%%%%%%%%%%%%%%%%%%%%%%%%%%%%%%%%%%%
Below we use some important properties of this set, see e.g. \cite[Ch.V]{Kato},
or \cite[Ch.1.6]{Zag}. Recall that ${\rm{dim}} ({\rm{ran}}(T))^{\bot}=:{\rm{def}}(T)$ is called a
\textit{deficiency} (or \textit{defect}) of a closed operator $T$ in $\mathfrak{H}$.
%%%%%%%%%%%%%%%%%%%%%%%%%%%% Proposition %%%%%%%%%%%%%%%%%%%%%%%%%%%%%%%%%%%%%%%%%%%%%%%%%%%%%%%
\begin{proposition}\label{prop-n-ran}
{\rm{(}}i{\rm{)}} Let $T$ be a closed operator in $\mathfrak{H}$. Then for any complex
number $z\notin \overline{{\mathfrak{N}}(T)}$, the operator $(T-z I)$ is
injective. Moreover, it has a closed range ${\rm{ran}}(T-z I)$ and
a constant deficiency ${\mathrm{def}}(T-z I)$ in each of connected component of
$\ {{\mathbb{C}}}\setminus{\overline{{\mathfrak{N}}(T)}}$.\\
{\rm{(}}ii{\rm{)}} If  ${\mathrm{def}}(T-z I) = 0$
for $z\notin \overline{{\mathfrak{N}}(T)}$, then $\Delta$ is a
subset of the resolvent set $\rho(T)$ of the operator $T$  and
\begin{equation}\label{estres}
\|(T-z I)^{-1}\| \leq {1\over {\rm{dist}}(z,\overline{{\mathfrak{N}}(T)})} \ \ .
\end{equation}
{\rm{(}}iii{\rm{)}} If ${\rm{dom}}(T)$ is dense and ${\mathfrak{N}}(T) \neq {\mathbb{C}}$, then
$T$ is closable, hence the adjoint operator $T^*$ is also densely defined.
\end{proposition}
%%%%%%%%%%%%%%%%%%%%%%%%%%%%%%%%%%%%%%%%%%%%%%%%%%%%%%%%%%%%%%%%%%%%%%%%%%%%%%%%%%%%%%%%%%%%%%%%
\begin{corollary}\label{spect-n-ran-bounded} For a bounded operator $T\in \mathcal{L}(\mathfrak{H})$ the spectrum
$\sigma(T)$ is a subset of $\ \overline{{\mathfrak{N}}(T)}$.
\end{corollary}
For unbounded operator $T$ the relation between spectrum and numerical range is more complicated.
For example, it may very well happen that $\sigma(T)$ is not contained
in $\overline{{\mathfrak{N}}(T)}$, but for a closed operator $T$ the essential spectrum
$\sigma_{ess}(T)$ is always a subset of $\overline{{\mathfrak{N}}(T)}$. The condition
${\mathrm{def}}(T-zI) = 0, \ z\notin \overline{{\mathfrak{N}}(T)}$ in Proposition \ref{prop-n-ran} (ii)
serves to ensure that for those unbounded operators one gets
\begin{equation}\label{spect-n-ran}
\sigma(T)\subset \overline{{\mathfrak{N}}(T)} \ ,
\end{equation}
i.e., the same conclusion as in Corollary \ref{spect-n-ran-bounded} for bounded operators.
%%%%%%%%%%%%%%%%%%%%%%%%%%%%%%%%%%%%%%%%%%%%%%%%%%%%%%%%%%%%%%%%%%%%%%%%%%%%%%%%%%%%%%%%%%%%%%%%%%
\begin{definition}\label{sect-oper}
{\em Operator $T$ is called \textit{sectorial} with {semi-angle}
$\alpha\in(0, \pi/2)$ and a vertex at $z=0$  if
$$
{\mathfrak{N}}(T)  \subseteq S_{\alpha} := \{z\in
{{\mathbb{C}}}: |\arg \ z| \leq \alpha \} \ .
$$
If, in addition, $T$ is closed and there is $z \in
{\mathbb{C}}\setminus S_{\alpha}$ such that it belongs to the resolvent set $\rho (T)$,
then operator $T$ is called  $m$-sectorial.}
\end{definition}
%%%%%%%%%%%%%%%%%%%%%%%%%%%%%%%%%%%%%%%%%%%%%%%%%%%%%%%%%%%%%%%%%%%%%%%%%%%%%%%%%%%%%%%%%%%%%%%%%%%%%%%%%%
\begin{remark}\label{rem-1-2}
Let $T$ be $m$-sectorial with the {semi-angle} $\alpha\in(0, \pi/2)$ and the vertex at $z=0$.
Then it is obvious that the operators $a T$ and $T_b := T + b$ belong to the same sector $S_{\alpha}$
for any non-negative parameters $a,b \geq 0$. In fact ${\mathfrak{N}}(T_b)\subseteq S_{\alpha} + b$, i.e.
the operator $T_b$ has the vertex at $z=b$.
\end{remark}
%%%%%%%%%%%%%%%%%%%%%%%%%%%%%%%%%%%%%%%%%%%%%%%%%%%%%%%%%%%%%%%%%%%%%%%%%%%%%%%%%%%%%%%%%%%%%%%%%%%%%%%%%%
Some of important properties of the $m$-sectorial operators are summarized by the following
%%%%%%%%%%%%%%%%%%%%%%%%%%%%%%%%%%%%%%%%%%%%%%%%%%%%%%%%%%%%%%%%%%%%%%%%%%%%%%%%%%%%%%%%%%%%%%%%%%%%%%%%%%%
\begin{proposition}\label{hol-contr}
If $T$ is {$m$-sectorial} in $\mathfrak{H}$, then the semigroup
$\{U(\zeta):= e^{-\zeta \, T}\}_{\, \zeta }$ generated by the operator $T${\rm{:}} \\
{\rm{(}}i{\rm{)}} is \textit{holomorphic} in the open sector $\{\zeta \in S_{\pi/2 - \alpha}\}${\rm{;}} \\
{\rm{(}}ii{\rm{)}} is a contraction, i.e. ${{\mathfrak{N}}}(U(\zeta))$
is a subset of the unit disc ${\mathfrak{D}}_{r=1}:=\{z\in {{\mathbb{C}}}: |z| \leq 1 \}$ for
$\{\zeta \in S_{\pi/2 - \alpha}\}$.
\end{proposition}
%%%%%%%%%%%%%%%%%%%%%%%%%%%%%%%%%%%%%%%%%%%%%%%%%%%%%%%%%%%%%%%%%%%%%%%%%%%%%%%%%%%%%%%%%%%%%%%%%%%%%%%%%%%
%%%%%%%%%%%%%%%%%%%%%%%%%%%%%%%%%%%%% Section 2 %%%%%%%%%%%%%%%%%%%%%%%%%%%%%%%%%%%%%%%%%%%%%%%%%%%%%%%%%%%
\section{Quasi-Sectorial Contractions and Main Theorem}\label{sec2}
\setcounter{equation}{0}
\renewcommand{\theequation}{\arabic{section}.\arabic{equation}}

The notion of the \textit{quasi-sectorial} contractions was introduced in \cite{CachZag-Q-S} to study the
operator-norm approximations of semigroups. In paper \cite{CaNeZag} this class of contractions appeared in
analysis of the operator-norm error bound estimate of the exponential Trotter product formula for the
case of accretive perturbations. Further applications of these contractions which, in particular, improve
the rate of convergence estimate of \cite{CachZag-Q-S} for the Euler formula, one can find in \cite{Paul},
\cite{Cach} and \cite{BentPaul}.
%%%%%%%%%%%%%%%%%%%%%%%%%%%%%%%%%%%%%%%%%%%%%%%%%%%%%%%%%%%%%%%%%%%%%%%%%%%%%%%%%%%%%%%%%%%%%%%%%%%%%%%%
\begin{definition}\label{alpha-domain}
{\em For $\alpha\in[0, \pi/2)$ we define in the complex plane ${\mathbb{C}}$ a \textit{closed domain}:
$$ D_{\alpha}:=\{z\in {\mathbb{C}}: |z|\leq \sin \alpha\} \cup
\{z\in {\mathbb{C}}: |\arg (1-z)|\leq \alpha \ {\rm{and}}\ |z-1|\leq
\cos \alpha \} .$$
This is a convex subset of the unit disc ${\mathfrak{D}}_{r=1}$, with \textit{"angle"}
(in contrast to \textit{tangent}) touching of its boundary $\partial {\mathfrak{D}}_{r=1}$
at only one point $z=1$, see Figure 1. It is evident that $D_{\alpha}\subset D_{\beta > \alpha}$.}
\end{definition}

%%%%%%%%%%%%%%%%%%%%%%%%%%%%%%%%%%%%%%%%%%%%%%%%%%%%%%%%%%%%%%%%%%%%%%%%%%%%%%%%%%%%%%%%%%%%%%%%%%%
\begin{definition}\label{Q-S-cont}{\rm{(}}Quasi-Sectorial Contractions {\rm{\cite{CachZag-Q-S})}}
{\em A {contraction} $C$ on the Hilbert space $\mathfrak{H}$ is called
\textit{quasi-sectorial} with  semi-angle $\alpha\in [0, \pi/2)$
with respect to the vertex at $z=1$, if
${\mathfrak{N}}(C) \subseteq D_{\alpha}$.}
\end{definition}

Notice that if operator $C$ is a \textit{quasi-sectorial} contraction, then
$I - C$ is an $m$-\textit{sectorial} operator with
vertex $z=0$ and semi-angle $\alpha$. The limits $\alpha=0$ and $\alpha = \pi/2$ correspond,
respectively, to non-negative (i.e. \textit{self-adjoint}) and to \textit{general} contraction.

The \textit{resolvent} of an $m$-sectorial operator $A$, with {semi-angle}
$\alpha\in(0, \pi/4]$ and vertex at $z=0$, gives the first non-trivial (and for us a \textit{key}) example
of a quasi-sectorial contraction.

\begin{proposition}\label{Th-res}
Let $A$ be $m$-sectorial operator with {semi-angle} $\alpha\in [0, \pi/4]$ and vertex at $z=0$.
Then $\{F(t):= (I + t A)^{-1}\}_{\, t \geq 0}$ is a family
of quasi-sectorial contractions which numerical ranges ${\mathfrak{N}}(F(t)) \subseteq D_{\alpha}$
for all $t \geq 0$.
\end{proposition}
\begin{proof}:
First, by virtue of Proposition \ref{prop-n-ran} (ii) we obtain the estimate:
\begin{equation}\label{F-contr}
\|F(t)\|\leq \frac{1}{t \ {\rm{dist}}(1/t \ , \, - S_{\alpha})} \, = \, 1 \ \ ,
\end{equation}
which implies that operators $\{F(t)\}_{\, t \geq 0}$ are contractions with numerical ranges
${\mathfrak{N}}(F(t)) \subseteq {\mathfrak{D}}_{r=1}$.

Next, by Remark \ref{rem-1-2} for all $u\in \mathfrak{H}$ one gets
$(u,F(t)u)= (v_t,v_t) + t (A v_t,v_t) \in S_{\alpha}$, where $v_t := F (t)u $, i.e. for any $t \geq 0$
the numerical range ${\mathfrak{N}}(F(t)) \subseteq S_{\alpha}$. Similarly, one finds that
$(u,(I-F(t))u) = t (v,A v) + t^2 (A v, A v) \in S_{\alpha}$, i.e.,
${\mathfrak{N}}(I-F(t)) \subseteq S_{\alpha}$. Therefore, for all $t\geq 0$ we obtain:
\begin{equation}\label{N-range-Fb}
{\mathfrak{N}}(F (t)) \subseteq (S_{\alpha} \cap (1 - S_{\alpha})) \subset \mathfrak{D}_{r=1} \ .
\end{equation}
Moreover, since $\alpha \leq \pi/4$, by Definition \ref{alpha-domain} we get
$(S_{\alpha} \cap (1 - S_{\alpha}))\subset D_{\alpha}$, i.e. for these values of $\alpha$ the
operators $\{F(t)\}_{\, t \geq 0}$ are quasi-sectorial contractions with numerical ranges in $D_{\alpha}$.
\end{proof}
%%%%%%%%%%%%%%%%%%%%%%%%%%%%%%%%%%%%%%%%%%%%%%%%%%%%%%%%%%%%%%%%%%%%%%%%%%%%%%%%%%%%%%%%%%%%%%%%%%%%%%%%%%%

Now we are in position to prove the \textit{main} Theorem establishing a relation between
quasi-sectorial contraction semigroups and a certain class of $m$-sectorial generators.
%%%%%%%%%%%%%%%%%%%%%%%%%%%%%%%%%%%%%%%%%%%%%%%%%%%%%%%%%%%%%%%%%%%%%%%%%%%%%%%%%%%%%%%%%%%%%%%%%%%%%%%%%%%
\begin{theorem}\label{Th-main}
Let $A$ be an {$m$-sectorial} operator with  semi-angle
$\alpha\in[0, \pi/4]$ and with  vertex at $z=0$. Then $\{e^{-t \, A}\}_{\, t \geq 0}$ is a quasi-sectorial
contraction semigroup with numerical ranges ${\mathfrak{N}}(e^{-t \, A}) \subseteq D_{\alpha}$ for all
$t \geq 0$.
\end{theorem}

%%%%%%%%%%%%%%%%%%%%%%%%%%%%%%%%%%%%%%%%%%%%%%%%%%%%%%%%%%%%%%%%%%%%%%%%%%%%%%%%%%%%%%%%%%%%%%%%%%%%%%%%%%%
The proof of the theorem is based on a series of lemmata and on the numerical range \textit{mapping}
theorem by Kato \cite{Kato-map} (see also an important comment about this theorem in \cite{Uch}).
%%%%%%%%%%%%%%%%%%%%%%%%%%%%%%%%%%%%%%%%%%%%%%%%%%%%%%%%%%%%%%%%%%%%%%%%%%%%%%%%%%%%%%%%%%%%%%%%%%%%%%%%%%%
\begin{proposition}\label{Kato-map}{\em \cite{Kato-map}}
Let $f(z)$ be a rational function on the complex plane ${\mathbb{C}}$, with $f(\infty)=\infty$.
Let for some compact and convex set $E'\subset {\mathbb{C}}$ the inverse function
$f^{-1}: E' \mapsto E \supseteq K $, where $K$ is a convex kernel of $E$, i.e., a subset of $E$
such that $E$ is \textit{star-shaped} relative to any $z\in K$.

If $C$ is an operator with numerical range $ {\mathfrak{N}}(C)\subseteq K$, then
${\mathfrak{N}}(f(C))\subseteq E'$.
\end{proposition}
Notice that for a convex set $E$ the corresponding  \textit{convex} kernel $K = E$.
%%%%%%%%%%%%%%%%%%%%%%%%%%%%%%%%%%%%%%%%%%%%%%%%%%%%%%%%%%%%%%%%%%%%%%%%%%%%%%%%%%%%%%%%%%%%%%%%%%%%%%%%%%%
\begin{lemma}\label{z-n}
Let $f_n(z) = z^n$ be complex functions, for $z\in {\mathbb{C}}$ and $n\in \mathbb{N}$.
Then the sets $f_n(D_\alpha)$ are convex and domains $f_n(D_\alpha)\subseteq D_\alpha$ for
{any $n\in \mathbb{N}$}, if $\alpha\leq \pi/4$.
\end{lemma}
%%%%%%%%%%%%%%%%%%%%%%%%%%%%%%%%%%%%%%%%%%%%%%%%%%%%%%%%%%%%%%%%%%%%%%%%%%%%%%%%%%%%%%%%%%%%%%%%%%%%%%%%%%%
\begin{lemma}\label{Euler}{\em(Euler formula)}
Let $A$ be an {$m$-sectorial} operator. Then for $t\geq 0$ one gets the strong limit
\begin{equation}\label{Eul-s-lim}
s-\lim_{n\rightarrow\infty}(F(t/n))^n = e^{-t A} \ .
\end{equation}
\end{lemma}
%%%%%%%%%%%%%%%%%%%%%%%%%%%%%%%%%%%%%%%%%%%%%%%%%%%%%%%%%%%%%%%%%%%%%%%%%%%%%%%%%%%%%%%%%%%%%%%%%%%%%%%%%%%
The next section is reserved for the proofs. They refine and modify some lines of reasonings
of the paper \cite{CachZag-Q-S}. This concerns, in particular, a corrected proofs of Proposition \ref{Th-res} and
Theorem \ref{Th-main} (cf. Theorem 2.1 of \cite{CachZag-Q-S}), as well as reformulations and proofs of
Propositions \ref{Kato-map} and Lemma \ref{z-n}.

%%%%%%%%%%%%%%%%%%%%%%%%%%%%%%%%%%%%%%%%%% Section 3 %%%%%%%%%%%%%%%%%%%%%%%%%%%%%%%%%%%%%%%%%%%%%%%%%%%%%%
\section{Proofs}\label{sec3}
\setcounter{equation}{0}
\renewcommand{\theequation}{\arabic{section}.\arabic{equation}}

\begin{proof}{\rm{(Lemma \ref{z-n})}}:\\
Let $\{z:|z|\leq \sin\alpha\}\subset D_\alpha$, then one gets $|z^n|\leq \sin\alpha$. Therefore, for the
mappings $f_n : z \mapsto z^n$ one obtains  $f_n (z) \in  D_\alpha $ for any $n \geq 1$.

Thus, it rests to check the same property only for images $f_n({\mathcal{G}}_{\alpha}), n\geq1$  of the sub-domain:
\begin{equation}\label{sub-dom}
{\mathcal{G}}_{\alpha}:=\!\{z: \!|\arg(z)|< \!(\pi/2 - \alpha) \}\cap\{z:\!|\arg(z + 1)|>(\pi - \alpha)\}\!
\subset \! D_\alpha ,
\end{equation}
see Definition \ref{alpha-domain} and Figure 1.

For $0\leq t\leq \cos\alpha$, two segments of tangent straight intervals:
$$\{\zeta_{\pm}(t)=1 + t \,e^{i (\pi \mp \alpha)}\}_{0\leq t\leq \cos\alpha} \subset \partial D_\alpha ,$$
are correspondingly \textit{upper} $\zeta_{+}(t)$ and  \textit{lower} $\zeta_{-}(t)= \overline{\zeta_{+}(t)}$
non-arc parts of the total boundary $\partial D_\alpha$; they also coincide with a part of
the boundary $\partial \mathcal{G}_{\alpha}$ connected to the vertex $z=1$.

Now we proceed by induction. Let $n=1$. Then one obviously obtain : $f_{n=1} (D_{\alpha}) = D_{\alpha}$.
For $n=2$ the boundary $\partial f_2 ({\mathcal{G}}_{\alpha})$ of domain $f_2 ({\mathcal{G}}_{\alpha})$ is a
union
$\Gamma_2 (\alpha)\cup \overline{\Gamma_2 (\alpha)}$ of the contour
$$\Gamma_2 (\alpha):= \{f_2 (\zeta_{+}(t))\}_{0\leq t\leq \cos\alpha}\cup \{z: |z|\leq \sin^2 \alpha , \
\arg (z)= (\pi - 2 \alpha)\}$$
and its conjugate $\overline{\Gamma_2 (\alpha)}$.
Since $\arg(\partial_t f_2 (\zeta_{+}(t))\leq (\pi - \alpha)$ for all $0\leq t\leq \cos\alpha$, the contour
$$\{f_2 (\zeta_{+}(t))\}_{0\leq t\leq \cos\alpha} \subseteq  \{z: |\arg(z + 1)|> (\pi - \alpha)\} ,$$
see (\ref{sub-dom}). The same is obviously true for the image of the lower branch $\zeta_{-}(t)$.
If $\alpha \leq \pi/4$,
one gets:
\begin{eqnarray}\label{Im+}
\sup_{0\leq t\leq \cos\alpha} {\rm{Im}} (f_2 (\zeta_{+}(t)))&=& {\rm{Im}} (f_2 (\zeta_{+}(t^* =
(2\cos\alpha)^{-1})))\\ &=&
\frac{1}{2} \tan \alpha < \sin\alpha \cos\alpha \,,\nonumber
\end{eqnarray}
where $t^* = (2\cos\alpha)^{-1} \leq \cos\alpha$, and
$$ 0 \geq {\rm{Re}}(f_2 (\zeta_{+}(t)))\geq -  \sin^2\alpha  \cos2\alpha \geq - \sin\alpha \,.$$
Therefore, $\{f_2 (\zeta_{+}(t))\}_{0\leq t\leq \cos\alpha} \subseteq D_\alpha$. Since the same is
also true for the image of the lower branch $\zeta_{-}(t)$, we obtain
$f_2 ({\mathcal{G}}_{\alpha})\subset D_\alpha$ and by consequence $f_{n=2} (D_\alpha)=
\{w = z\cdot z :  \ z\in D_\alpha,  \ z \in f_{n=1} (D_\alpha)\} \subset D_\alpha ,$ for $\alpha \leq \pi/4$.

Now let $n>2$ and suppose that $f_n (D_\alpha)\subset D_\alpha$. Then the image of the $(n+1)\ -$ order mapping
of domain $D_\alpha$ is:
$$f_{n+1}(D_\alpha)= \{w = z\cdot z^n :  \ z\in D_\alpha,  \ z^n \in f_n (D_\alpha)\}, $$
and since $f_n (D_\alpha)\subset D_\alpha$, we obtain $f_{n+1}(D_\alpha)\subset D_\alpha$ by the same reasoning
as for $n=2$.
\end{proof}
%%%%%%%%%%%%%%%%%%%%%%%%%%%%%%%%%%%%%%%%%%%%%%%%%%%%%%%%%%%%%%%%%%%%%%%%%%%%%%%%%%%%%%%%%%%%%%%%%%%%%%%%%%
\begin{remark}\label{max-Im}
Let $\phi(t):= \arg(\zeta_{+}(t))$. Then $\cot(\alpha + \phi(t))= (\cos\alpha - t)/\sin\alpha$ and
\begin{equation}\label{estim-Im}
\sup_{0\leq t\leq \cos\alpha}{\rm{Im}}(f_n (\zeta_{+}(t)))\leq (1 -2t_{n}^*\cos\alpha + (t_{n}^*)^2)^{n/2}
\end{equation}
for $\sin (n\phi(t_{n}^*))=1$. In the limit $n\rightarrow\infty$ this implies that
$\phi(t_{n}^*) = \pi/2n + o(n^{-1})$, $t_{n}^* = \pi/(2n \sin\alpha)  + o(n^{-1})$ and
\begin{equation}\label{Im-lim}
\lim_{n\rightarrow\infty}\sup_{0\leq t\leq \cos\alpha}{\rm{Im}}(f_n (\zeta_{+}(t)))
\leq \exp(- \frac{1}{2}\pi \cot\alpha) < \frac{1}{2} \tan \alpha .
\end{equation}
By the same reasoning one gets the estimates similar to (\ref{estim-Im}) and (\ref{Im-lim}) for $\zeta_{-}(t))$.
Hence, $|{\rm{Im}}(f_n (\zeta_{\pm}(t)))|< {\rm{Im}}(f_{n=1}(\zeta_{+}(t)))< \sin\alpha \cos\alpha$,
cf. (\ref{Im+}).

Notice that in spite of the arc-part of the contour
$\partial{D_\alpha}$ shrinks in the limit $n\rightarrow\infty$ to zero, we obtain
\begin{equation}\label{Re-lim}
\lim_{n\rightarrow\infty}\sup_{0\leq t\leq \cos\alpha}{\rm{Re}}(f_n (\zeta_{+}(t))) =
- \exp(- \pi \cot\alpha) ,
\end{equation}
for the left extreme point of the projection on the real axe ($\sin (n\phi(t_{n}^*))=1$) of the image $f_n (D_\alpha)$.
Since $\exp(- \pi \cot\alpha)< \sin\alpha$, for $\alpha \leq \pi/4$, the arguments (\ref{Im-lim}) and (\ref{Re-lim})
bolster the conclusion of the Lemma \ref{z-n}.
\end{remark}
%%%%%%%%%%%%%%%%%%%%%%%%%%%%%%%%%%%%%%%%%%%%%%%%%%%%%%%%%%%%%%%%%%%%%%%%%%%%%%%%%%%%%%%%%%%%%%%%%%%%%%%%%%
\begin{proof}{\rm{(Lemma \ref{Euler})}}:\\
By (\ref{F-contr}) we have for $\lambda > 0$
\begin{equation}\label{H-Y-est}
\|(\lambda I + A)^{-1}\| < \lambda^{-1} \ ,
\end{equation}
and since $A$ is $m$-sectorial, we also get that $(-\infty, 0) \subset \rho(A)$.
Then the \textit{Hille-Yosida} theory ensures the existence of the contraction semigroup
$\{e^{-t \, A}\}_{\, t \geq 0}$,
and the standards arguments (see e.g. \cite[Ch.V]{Kato}, or \cite[Ch.1.1]{Zag})
yield the convergence of the
Euler formula (\ref{Eul-s-lim}) in the strong topology.
\end{proof}
%%%%%%%%%%%%%%%%%%%%%%%%%%%%%%%%%%%%%%%%%%%%%%%%%%%%%%%%%%%%%%%%%%%%%%%%%%%%%%%%%%%%%%%%%%%%%%%%%%%
\begin{proof}{\rm{(Theorem \ref{Th-main})}}:\\
Take $f(z) = z^2$ and the compact convex set $E':=  f(D_{\alpha})\subseteq D_{\alpha}$, see Lemma \ref{z-n}.
Since the set $E:= f^{-1}(E')= D_{\alpha} \cup (-D_{\alpha})$ is \textit{convex}, its convex kernel $K$
exists and $K = E$. Then by Proposition \ref{Kato-map} we obtain that
${\mathfrak{N}}(f(C))\subseteq E'\subseteq D_\alpha$, if the numerical range ${\mathfrak{N}}(C)\subseteq K$.

Let contraction $C_1 := (I + t \,A /2)^{-1}= F(t/2)$. Since by Proposition \ref{Th-res}
for any $t\geq0$ we have ${\mathfrak{N}}(C_1)\subseteq D_{\alpha}$ and since $D_{\alpha}\subset E$,
we can choose  $K = E$. Then by the Kato numerical range mapping theorem (Proposition \ref{Kato-map}) we get:
\begin{equation}\label{1st-step}
{\mathfrak{N}}(f(C_1)= F (t/2)^2)\subseteq E'\subseteq D_{\alpha}  \ .
\end{equation}
Similarly, take the contraction $C_2 := F (t/4)^2$. Since (\ref{1st-step}) is valid for any $t\geq 0$,
it is true for $t \mapsto t/2$.  Then by definition of $K$ one has
${\mathfrak{N}}(F (t/4)^2)\subseteq D_{\alpha} \subseteq K$. Now
again the Proposition \ref{Kato-map} implies:
\begin{equation}\label{2st-step}
{\mathfrak{N}}(f(C_2)= F(t/4)^4)\subseteq E'\subseteq D_{\alpha} \ .
\end{equation}
Therefore, we obtain ${\mathfrak{N}}(F_b(t/2^n)^{2^n})\subseteq D_{\alpha}$, for any $n \in {\mathbb{N}}$.
By Lemma \ref{Euler} this yields
$$
\lim_{n\rightarrow\infty}(u,(I + t \,A /2^n)^{-2^n} u) = (u,e^{- t \, A} u) \in D_{\alpha} \ ,
$$
for any unit vector $u\in\mathfrak{H}$. Therefore, the numerical ranges of the contraction semigroup
${\mathfrak{N}}(e^{- t \, A})\subseteq D_{\alpha}$ for all $t\geq 0$, if it is generated by {$m$-sectorial}
operator with  the semi-angle $\alpha\in[0, \pi/4]$ and with the vertex at $z=0$.
\end{proof}

%%%%%%%%%%%%%%%%%%%%%%%%%%%%%%%%%%%%%%%%% Corollaries and Applications %%%%%%%%%%%%%%%%%%%%%%%%%%%%%%%%%%%%%%%%%
\section{Corollaries and Applications}
\setcounter{equation}{0}
\renewcommand{\theequation}{\arabic{section}.\arabic{equation}}

1. Notice that Definition \ref{Q-S-cont} of \textit{quasi-sectorial} contractions $C$ is quite
{restrictive} comparing to the notion of \textit{general} contractions, which demands \textit{only}
${\mathfrak{N}}(C)\subseteq \mathfrak{D}_1$. For the latter case one has a well-known
\textit{Chernoff lemma} \cite{Chern-1}:
\begin{equation}\label{Chern-Lemm}
\|(C^n - e^{n(C-I)})u\| \leq n^{1/2}\|(C-I)u\| \ , \ u\in\mathfrak{H} \ , \ n \in \mathbb{N} \ ,
\end{equation}
which is \textit{not} even a convergent bound.
For quasi-sectorial contractions we can obtain a much stronger estimate \cite{CachZag-Q-S}:
\begin{equation}\label{n-1/3}
\left\|C^n - e^{n(C-I)}\right\| \leq {M \ n^{-1/3}} \ \  , \ \ n \in \mathbb{N} \ ,
\end{equation}
convergent to zero in the uniform topology when $n \rightarrow \infty$.
Notice that the rate of convergence $n^{-1/3}$ obtained in \cite{CachZag-Q-S} with
help of the \textit{Poisson representation} and the
\textit{Tchebychev inequality} is not optimal. In \cite{Paul},
\cite{Cach} and \cite{BentPaul} this estimate was improved up to the
\textit{optimal} rate $O(n^{-1})$, which one can easily verify for a particular
case of self-adjoint contractions (i.e. $\alpha=0$) with help of the  spectral
representation.
%%%%%%%%%%%%%%%%%%%%%%%%%%%%%%%%%%%%%%%%%%%%%%%%%%%%%%%%%%%%%%%%%%%%%%%%%%%%%%%%%%%%%%%%%%%%%%%%%%%%%%%%%
%%%%%%%%%%%%%%%%%%%%%%%%%%%%%%%%%%%%%%%%%%%%%%%%%%%%%%%%%%%%%%%%%%%%%%%%%%%%%%%%%%%%%%%%%%%%%%%%%%%%%%%%%

The inequality (\ref{n-1/3}) and its further improvements are based on the following
important result about the upper bound estimate for the case of \textit{quasi-sectorial} contractions:
%%%%%%%%%%%%%%%%%%%%%%%%%%%%%%%%%%%%%%%%%%%%%%%%%%%%%%%%%%%%%%%%%%%%%%%%%%%%%%%%%%%%%%%%%%%%%%%%%%%%%%%%%
\begin{proposition}\label{C-n-estim}
If $C$ is a quasi-sectorial contraction on a Hilbert space $\mathfrak{H}$ with semi-angle
$0\leq\alpha < \pi/2$, i.e. the numerical range
${\mathfrak{N}}(C)$ is a subset of the domain $D_\alpha$, then
\begin{equation}\label{C1}
\|C^n (I-C)\|\leq \frac{K}{n+1} \ , \ n\in{\mathbb{N}} \ .
\end{equation}
\end{proposition}
%%%%%%%%%%%%%%%%%%%%%%%%%%%%%%%%%%%%%%%%%%%%%%%%%%%%%%%%%%%%%%%%%%%%%%%%%%%%%%%%%%%%%%%%%%%%%%%%%%%%%%%%
\noindent For the proof see Lemma 3.1 of \cite{CachZag-Q-S}.

2. Another application of quasi-sectorial contractions generalizes the Chernoff semigroup approximation
theory \cite{Chern-1}, \cite{Chern-2} to the operator-norm approximations \cite{CachZag-Q-S}.
%%%%%%%%%%%%%%%%%%%%%%%%%%%%%%%%%%%%%%%%%%%%%%%%%%%%%%%%%%%%%%%%%%%%%%%%%%%%%%%%%%%%%%%%%%%%%%%%%%%%%%%%
\begin{proposition}\label{norm-appr}
Let $\{\Phi(s)\}_{s\geq 0}$ be a family of uniformly
quasi-sectorial contractions on a Hilbert space $\mathfrak{H}$, i.e. such that there exists
$0<\alpha<\pi/2$ and  ${\mathfrak{N}}(\Phi(s)) \subseteq D_\alpha$, for all $s\geq 0$. Let
$$
X(s):=(I-\Phi(s))/s  \ ,
$$
and let $X_0$ be a closed operator with non-empty resolvent set, defined in a closed subspace
${\mathfrak{H}}_0 \subseteq \mathfrak{H}$. Then the family $\{X(s)\}_{s>0}$ converges,
when $s\rightarrow +0$, in the
uniform resolvent sense to the operator $X_0$  if and only if
\begin{equation}\label{lim}
\lim_{n\rightarrow \infty} \left\|\Phi(t/n)^n -e^{-tX_0}P_0\right\| = 0 \ , \ \ \ \mbox{for} \ t>0 \ .
\end{equation}
Here $P_0$ denotes the orthogonal projection onto the subspace ${\mathfrak{H}}_0$.
\end{proposition}
%%%%%%%%%%%%%%%%%%%%%%%%%%%%%%%%%%%%%%%%%%%%%%%%%%%%%%%%%%%%%%%%%%%%%%%%%%%%%%%%%%%%%%%%%%%%%%%%%%%%%%%%%%%%%%

3. We conclude by application of Theorem \ref{Th-main} and Proposition \ref{C-n-estim}
to the Euler formula \cite{CachZag-Q-S}, \cite{Paul}, \cite{Cach}.
\begin{proposition}\label{norm-appr-Euler}
If $A$ is an $m$-sectorial operator in a Hilbert space
${\mathfrak{H}}$, with semi-angle $\alpha\in [0,\pi/4]$ and with vertex at $z=0$, then
$$
\lim_{n\rightarrow\infty}\left\|(I+tA/n)^{-n} - e^{-tA}\right\| =
0,\ t\in S_{\pi/2-\alpha}.
$$
Moreover, uniformly in $t\geq t_0 > 0$ one has the error estimate:
$$
\left\|(I+tA/n)^{-n} - e^{-tA}\right\| \leq
O\left(n^{-1}\right) \ , \ n\in{\mathbb{N}} \ .
$$
\end{proposition}
%%%%%%%%%%%%%%%%%%%%%%%%%%%%%%%%%%%%%%%%%%%%%%%%%%%%%%%%%%%%%%%%%%%%%%%%%%%%%%%%%%%%%%%%%%%%%%%%%%%%%%%%%%%%

\section*{Acknowledgements}
I would like to thank Professor Mitsuru Uchiyama for a useful remark indicating a
flaw in our arguments in Section 2 of \cite{CachZag-Q-S} , revision of
this part of the paper \cite{CachZag-Q-S} is done in the present manuscript.
I also thankful to Vincent Cachia and Hagen Neidhardt for a pleasant collaboration.

%%%%%%%%%%%%%%%%%%%%%%%%%%%%%%%%%%%%%%%%% Biblio %%%%%%%%%%%%%%%%%%%%%%%%%%%%%%%%%%%%%%%%%%%%%%%%%%%%%%%%%%

%%%%%%%%%%%%%%%%%%%%%%%%%%%%%%%%%%%%%%%%%%%%%%%%%%%%%%%%%%%%%%%%%%%%%%%%%%%%%%%%%%%%%%%%%%%%%%%%%%%%%%%%%
%%%%%%%%%%%%%%%%%%%%%%%%%%%%%%%%%%%%%%%%%%%%%%%%%%%%%%%%%%%%%%%%%%%%%%%%%%%%%%%%%%%%%%%%%%%%%%%%%%%%%%%%%
%%%%%%%%%%%%%%%%%%%%%%%%%%%%%%%%%%%%%%%%%%%%%%%%%%%%%%%%%%%%%%%%%%%%%%%%%%%%%%%%%%%%%%%%%%%%%%%%%%%%%%%%%
\begin{figure}[H]
\epsfysize=10cm

$$\epsfbox{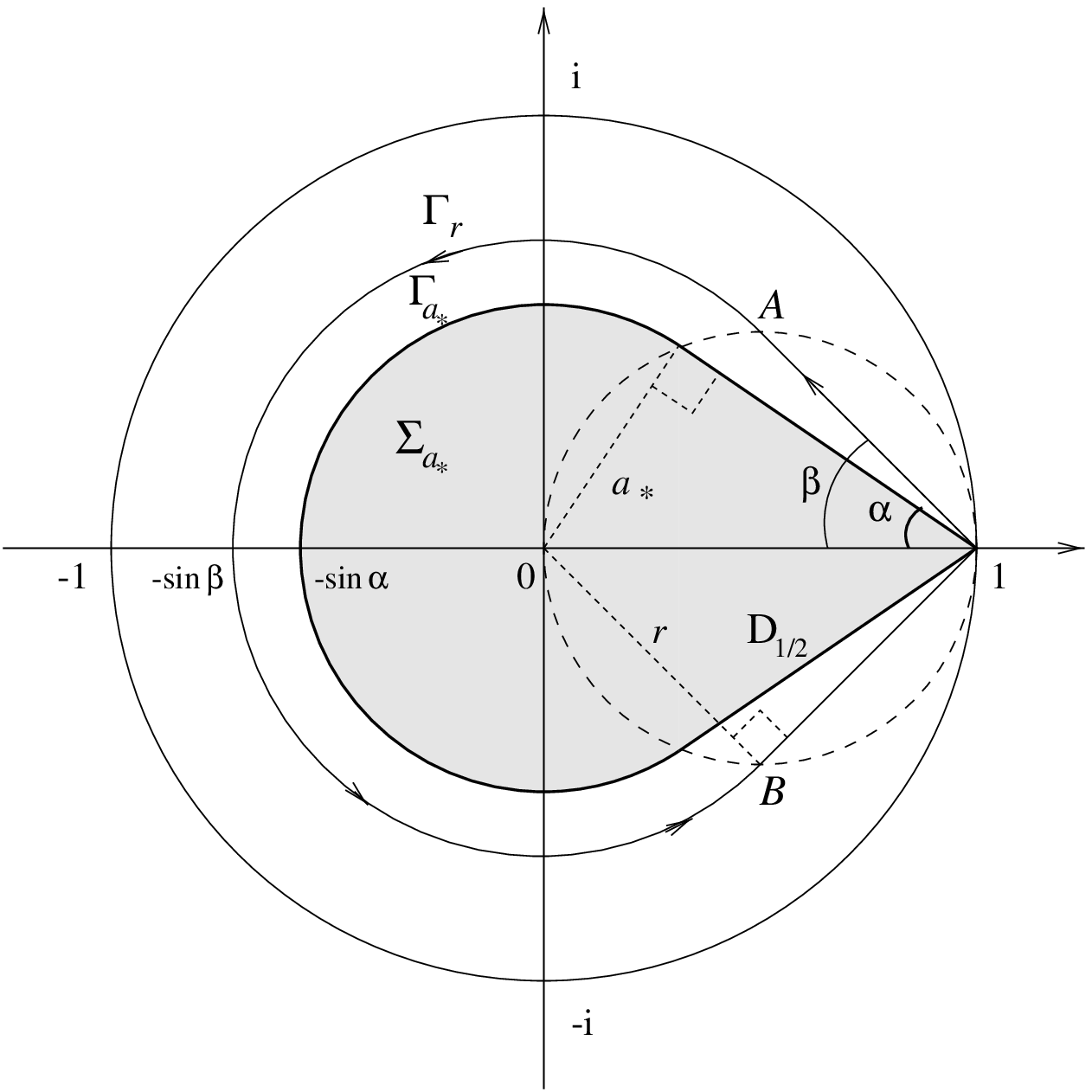}$$
\caption{Illustration of the set $D_{\alpha}$($= \Sigma_{a_*}$ shaded domain) with boundary 
$\partial D_{\alpha} = \Gamma_{a_*}$, where $a_* = \sin \alpha$, as well as of our choice of the contour 
$\Gamma_r$ in the resolvent set $\rho(C)$, where $r=\sin \beta >a_*$. The contour $\Gamma_r$ consists of two 
segments of tangent straight lines $(1,A)$ and $(1,B)$ and the arc $(A,B)$ of radius $r$. 
The dotted circle $\partial \mbox{D}_{r=1/2}$ corresponds to the set of tangent points for different 
values of $\alpha\in[0,\pi/2]$.}

\end{figure}
%%%%%%%%%%%%%%%%%%%%%%%%%%%%%%%%%%%%%%%%%%%%%%%%%%%%%%%%%%%%%%%%%%%%%%%%%%%%%%%%%%%%%%%%%%%%%%%%%%%%%%%%%
%%%%%%%%%%%%%%%%%%%%%%%%%%%%%%%%%%%%%%%%%%%%%%%%%%%%%%%%%%%%%%%%%%%%%%%%%%%%%%%%%%%%%%%%%%%%%%%%%%%%%%%%%%
%%%%%%%%%%%%%%%%%%%%%%%%%%%%%%%%%%%%%%%%%%%%%%%%%%%%%%%%%%%%%%%%%%%%%%%%%%%%%%%%%%%%%%%%%%%%%%%%%%%%%%%%%%

\end{document}